\documentclass[11pt]{article}
\usepackage{amsfonts,graphicx}
\usepackage{amsmath,amssymb}
\usepackage{amsbsy,amsmath}

\def\R{\mathbb{R}}

\def\N{\mathbb{N}}

\def\cA{\mathcal{A}}

\def\cD{\mathcal{D}}

\def\cE{\mathcal{E}}

\def\cP{\mathcal{P}}

\def\aa{{\boldsymbol{a}}}

\def\cc{\boldsymbol{c}}

\def\ee{\boldsymbol{e}}

\def\ii{\boldsymbol{i}}
\def\jj{\boldsymbol{j}}

\def\eell{\boldsymbol{\ell}}

\def\pp{\boldsymbol{p}}

\def\xx{\boldsymbol{x}}
\def\yy{\boldsymbol{y}}

\def\NN{\boldsymbol{N}}

\def\00{\boldsymbol{0}}
\def\11{\boldsymbol{1}}
\def\22{\boldsymbol{2}}
\def\33{\boldsymbol{3}}
\def\44{\boldsymbol{4}}
\def\55{\boldsymbol{5}}
\def\66{\boldsymbol{6}}
\def\77{\boldsymbol{7}}
\def\88{\boldsymbol{8}}
\def\99{\boldsymbol{9}}

\def\aalpha{{\boldsymbol{\alpha}}}

\def\ggamma{{\boldsymbol{\gamma}}}
\def\ddelta{{\boldsymbol{\delta}}}

\def\mfu{{{\mathfrak{u}}}}

\def\mfv{{{\mathfrak{v}}}}

\DeclareMathOperator*{\argmin}{arg\,min}

\DeclareMathOperator{\vol}{vol}

\def\mix{{\qopname\relax o{mix}}}

\def\mybigtimes{\mathop{\mathchoice{
   \vcenter{\hbox to10bp{\vrule height15bp width0pt \pdfliteral{
   q 1 J .8 w 0 1 m 10 14 l S 0 14 m 10 1 l S Q
}\hss}}}{
   \vcenter{\hbox to10bp{\kern1bp\vrule height10bp width0pt \pdfliteral{
   q 1 J .65 w 0 0 m 8 10 l S 0 10 m 8 0 l S Q
}\hss}}}{\times}{\times}
}}

\def\transpose{{\rm T}}

\usepackage{color}

\usepackage{stmaryrd}
\DeclareMathOperator*{\bigtimes}{\vartimes}

\newtheorem{theorem}{Theorem}[section]
\newtheorem{corollary}[theorem]{Corollary}
\newtheorem{lemma}[theorem]{Lemma}
\newtheorem{proposition}[theorem]{Proposition}
\newtheorem{definition}[theorem]{Definition}

\def\qed{\hfill$\Box$} 
\newenvironment{proof}{{\noindent \bf 
    Proof:}}{\hfill\qed\bigskip}

\begin{document}

\title{Improved Sampling Inequalities for Sparse Grids and
  High-Dimensional Functions with Effective Low Dimension}
\author{Christian Rieger\\ FB Mathematik und
  Informatik\\ Philipps-Universit\"at Marburg \\ 35032 Marburg\\Germany \and Holger Wendland \\Department of
  Mathematics \\ University of Bayreuth\\ 95440 Bayreuth\\ Germany
}

\maketitle
\begin{abstract}
  The approximation of high-dimensional functions is a challenging task
  due to the often appearing curse of dimensionality. In this paper,
  we combine sparse grid with anchored projection techniques to derive
  sampling inequalities for Sobolev functions of a dominating mixed
  regularity which are effectively low dimensional. To this end, we
  derive new   sampling inequalities for sparse grids and combine
  these with recently investigated regression processes of
  non-matching sampling processes.
\end{abstract}

\section{Introduction}
Sampling inequalities play a crucial role in providing error estimates
for multivariate approximation processes. They have been introduced in
\cite{Madych-Potter-85-1} for grid data and in
\cite{Wendland-Rieger-05-1} for scattered data based on earlier work
in \cite{Narcowich-etal-05-1}. Improved bounds can be found in
\cite{Arcangeli-etal-07-1,Arcangeli-etal-12-1, Narcowich-etal-06-1, Madych-06-1}. 
These sampling inequalities are all using the so-called fill-distance
to measure how well a given point set covers the underlying domain and
are hence only relevant for low-dimensional problems. For higher
dimensional problems, sparse grids, see for example
\cite{Bungartz-Griebel-04-1}, and approximation in Sobolev spaces of dominating
mixed regularity are a well-known and well-established
concept. Another often employed approach to tackle high-dimensional problems is
the use of an ANOVA (see for example
\cite{Gelman-05-1,Sobol-01-1,Owen-13-1}) or anchored decomposition
(see for example \cite{Goda-Dick-15-1,Kuo-etal-10-1}) which also
appear under the names  
many body expansions or high-dimensional model representation, see for
example \cite{Rabitz-Alics-99-1} and also \cite{Griebel-06-1} for a
more mathematical overview. 

One goal of this paper is to generalize and improve previous results
on sampling inequalities for functions from mixed regularity Sobolev
spaces over sparse grids.
The first prototypes of such sampling inequalities can be found in
\cite{Rieger-Wendland-17-1, Rieger-Wendland-20-1}. While the first
paper deals with isotropic sparse grids, the second generalizes the
results to anisotropic ones. In both cases the underlying univariate
point sets are based on Chebyshev or Clenshaw-Curtis points, which are
suited for polynomial interpolation but they are not quasi-uniform points,
which are more common when it comes to kernel-based interpolation and
approximation. Hence, our first aim is to generalize these results for
sparse grids build upon quasi-uniform points. We will do this only for
isotropic sparse grids but a generalization to anisotropic ones will
be straight-forward.  In doing this, we also improve the exponents in
the logarithmic terms, which are now in line with other more specific
error estimates, though they are not yet optimal in the sense of
\cite{Dolbeault-etal-23-1}. In contrast to the results in
\cite{Dolbeault-etal-23-1}, which are non constructive and only show
that there are point sets that satisfy a better bound, our bounds hold
for all sparse grids based upon quasi-uniform sets.
Moreover, as we are particularly interested in
$L_2$-estimates, we also measure the discrete error on the sampling
points in the discrete $\ell_2$-norm.

The second main goal of this paper is to to combine the newly derived
sampling inequalities with the approximation of high-dimensional
functions by truncated anchored decompositions. This continues and
improves our previous work in \cite{Rieger-Wendland-24-1} and
particularly in \cite{Rieger-Wendland-26-1} and has deep applications
to problems from uncertainty quantification and parametric partial
differential equations. 

Based on the goals stated above, the outline of this paper is as
follows. In the next section, we will collect and proof some necessary
univariate results. Section \ref{sec:samponsparse} contains our first
main results on improved sampling inequalities on sparse grids. Section
\ref{sec:samphighwithlow} is devoted to sampling inequalities on
specific data sets for specific high-dimensional functions, which have
an anchored representation using only low-dimensional terms. 
In the last section, the results of the previous two sections are then
combined to state and prove the desired sampling inequalities for
high-dimensional functions with an effective low dimension.

\section{Univariate Discussion}

We are concerned with univariate quasi-interpolation processes which
form local polynomial reproductions. Such processes, even in the
multivariate setting, were introduced in \cite{Wendland-01-1}, see
also \cite{Wendland-05-1}.

Throughout this paper, let $I\subseteq\R$ be a closed interval. For
$m\in\N$ let $\pi_m(\R)$ denote the space of univariate polynomials of
degree up to $m$ and $\pi_m(I)=\pi_m(\R)|I$. Moreover, let $H^m(I)$ be
the Sobolev space of all $u\in L_2(I)$ with weak derivatives
$u^{(j)}\in L_2(I)$ for $1\le j\le m$ equipped with the semi-norm
$|f|_{H^m(I)}=\|f^{(m)}\|_{L_2(I)}$ and norm
$\|f\|_{H^m(I)}^2=\sum_{j=0}^m \|f^{(j)}\|_{L_2(I)}^2$. 

For a point set $X=\{x_1,\ldots,x_N\}\subseteq I$ we define the fill distance $h_X$, the
separation radius $q_X$ and its ratio $\rho_X=h_X/q_X$ by
\[
h_{X}=\sup_{x\in I}\min_{1\le j\le N} |x-x_j|, \qquad
q_X = \frac{1}{2} \min_{j\ne k} |x_j-x_k|.
\]

\begin{definition}
Assume that for any  $X=\{x_1,\ldots,x_N\}\subseteq I$  there are
functions $a_j:I\to\R$ for $1\le i\le N$. The associated {\em
  quasi-interpolation operator} $A_X$ is then defined via
\[
A_X f(x):=\sum_{j=1}^N a_j(x) f(x_j), \qquad x\in I, f\in C(I).
\]
The process that maps each finite set $X\subseteq I$ to $A_X$ is
called a local polynomial reproduction of degree $m\in\N_0$ if there are constants
$h_0,c_1,c_2>0$ such that
\begin{enumerate}
\item $A_Xp=p$ for all $p\in\pi_m(I)$, \label{lpr1}
\item $\displaystyle \sum_{j=1}^N |a_j(x)|\le c_1$ for all $x\in I$,\label{lpr2}
\item $a_j(x)=0$ whenever $|x-x_j|>c_2h_X$ \label{lpr3}
\end{enumerate}
holds for all finite $X\subseteq I$ with $h_X\le h_0$.
\end{definition}

The existence of such local polynomial reproduction processes is
well-known, see the references above. As a matter of fact, the functions $a_j$ can be chosen
arbitrarily smooth, see for example \cite{Hangelbroek-Rieger-25-1}. 

We need the following two auxiliary results on local polynomial
reproduction processes. The first one deals with the approximation
power of quasi-interpolation processes and can be proven using
techniques as in the proof of \cite[Theorem 3.2]{Wendland-05-1} in
combination with techniques from \cite{Duchon-78-1}.

\begin{lemma}\label{lem:errorbound1d}
Suppose $X\mapsto A_X$ is a local polynomial reproduction process of
degree $m-1$ with $m\in \N$. Then, there is a constant $C>0$ such that
\[
\|f-A_Xf\|_{L_2(I)} \le C h_X^m \|f\|_{H^m(I)}
\]
for all $f\in H^m(I)$ and all finite $X\subseteq I$ with $h_X\le h_0$.
\end{lemma}
\begin{proof}
 We start by assuming $f\in C^m(I)$ and setting $h:=h_X$. Then, for any $x\in I$
we define $I_x:=[x-c_2h,x+c_2h]\cap I$. With
 any $p\in\pi_{m-1}(I) $ we can bound the point-wise error using the
 three properties of a local polynomial reproduction by
 \begin{eqnarray*}
   |f(x)-A_Xf(x)| & \le& |f(x)-p(x)| + \sum_{j=1}^N |f(x_j)-p(x_j)|
   |a_j(x)| \\
   &\le &\left(1+\sum_{j=1}^N |a_j(x)| \right)
   \|f-p\|_{L_\infty(I_x)}\\
   & \le & (1+c_1) \|f-p\|_{L_\infty(I_x)}.
 \end{eqnarray*}
 Next, using $p$ as the Taylor polynomial to $f$ of degree $m-1$ about
 $x$, we find for any $y\in I_x$ the bound
 \begin{eqnarray*}
   |f(y)-p(y)| & \le & \frac{1}{(m-1)!} \left|\int_x^y (y-t)^{m-1}
   f^{(m)}(t) dt \right|\\
   & \le & \frac{c_2^m h^{m-1}}{(m-1)!}\sqrt{|y-x|}
   \left(\int_x^y|f^{(m)}(t)|^2 dt\right)^{1/2}\\
   & \le & c_m h^{m-1/2} |f|_{H^m(I_x)},
 \end{eqnarray*}
 showing that we have a point-wise error bound of the form
 \[
 |f(x)-A_Xf(x)| \le c h^{m-1/2} |f|_{H^m(I_x)} , \qquad x\in I.
 \]
 Next, we cover $I$ by disjoint intervals $I_j$, $1\le j\le L$, of length $2c_2h$, where
 the last interval might be shorter. Thus,
 for any $x\in I$, the interval $I_x$ is contained in the union of at most two consecutive
 intervals. Moreover, for $x\in I_j$, we have $I_x\subseteq
 I_{j-1}\cup I_j\cup I_{j+1}$, where we set $I_0=I_{L+1}=\emptyset$. This then leads to
 \begin{eqnarray*}
   \|f-A_Xf\|_{L_2(I)}^2 & = & \sum_{j=1}^L \int_{I_j} 
   |f(x)-A_Xf(x)|^2 dx  \le c h^{2m-1}  \sum_{j=1}^L \int_{I_j} \int_{I_x}
   |f^{(m)}(y)|^2 dy dx \\
   &\le & ch^{2m-1}\sum_{j=1}^L \int_{I_j} \int_{I_{j-1}\cup I_{j}\cup I_{j+1}}
   |f^{(m)}(y)|^2 dy dx \\
   & \le & ch^{2m-1} 3 h |f|_{H^m(I)}^2,
 \end{eqnarray*}
 which settles the case for $C^m(I)$ functions. The general case of
 Sobolev functions $f\in H^m(I)$ then follows by density, as usual.
\end{proof}

So far, we have defined $A_X$ as an operator acting on functions from
$C(I)$. In the above lemma, for example, we interpreted it as a
mapping $A_X:C(I)\to L_2(I)$. However, for any $X$ with $N$ points,
we can also define it as an operator acting on vectors
$\cc=(c_1,\ldots,c_N)^\transpose\in\R^N$. Formally, we can achieve this by defining a new operator
$\widetilde{A}_X:\R^N\to L_2(I)$ by 
\[
\widetilde{A}_X\cc (x):=\sum_{j=1}^N a_j(x) c_j, \qquad x\in I,
\]
and a sampling operator $S_X:C(I)\to \R^N$  by $S_X
f = (f(x_1),\ldots,f(x_N))^\transpose$, showing 
\[
A_Xf = (\widetilde{A}_X\circ S_X) f.
\]

For the following result we only need Properties \ref{lpr2} and
\ref{lpr3} of a local polynomial reproduction process. The polynomial
reproduction is not required.

\begin{lemma}\label{lem:discrete}
Suppose $X\mapsto A_X$ is a local polynomial reproduction process of
degree $m\in\N_0$.  Then,
\[
\|\widetilde{A}_X\cc\|_{L_2(I)} \le c_1 (2c_2+1) (\rho_X h_X)^{1/2}
\|\cc\|_2, \qquad \cc\in \R^N,
\]
provided that $h_X\le h_0$.
\end{lemma}
\begin{proof}
  For $\cc\in\R^N$ we observe
  \[
  \|\widetilde{A}_X\cc\|_{L_2(I)}^2 = \sum_{j,k=1}^N c_j c_k \int_I
  a_j(x)a_k(x) dx = \cc^\transpose M \cc \le \|M\|_{2,2}\|\cc\|_2^2,
  \]
  with the mass matrix $M=(m_{jk})$ with $m_{jk}= \int a_j a_k
  dx$. To
  bound this further, we first note that the symmetric matrix $M$ is actually
  sparse, as we have $m_{jk}=0$ whenever $|x_j-x_k| >2c_2h_X$. Hence,
  the number of non-zero entries in row (or column) $j$ is bounded by
  \[
  |\{k\in\{1,\ldots,N\} : |x_j-x_k|\le 2c_2 h_X\}| \le
  2c_2\frac{h_X}{q_X}+1 \le (2c_2+1)\rho_X,
  \]
  using a standard volume comparison argument, the fact that we are
  one dimensional and $1\le \rho_X$. Moreover, using the boundedness
  of the Lebesgue constants and the locality of the $a_j$ again,
  yields with $I_j:=[x_j-c_2h_X,x_j+c_2h_X]\cap I$, 
  \[
  |m_{jk}| = \left|\int_I a_j(x) a_k(x) dx \right| \le
  \int_{I_j} c_1^2 dx \le 2c_1^2c_2h_X \le c_1^2(2c_2+1) h_X.
  \]
  As $M$ is symmetric this yields
  \[
  \|M\|_{\infty,\infty} = \|M\|_{1,1} = \max_{1\le j\le n}
  \sum_{k=1}^N |m_{jk}| \le c_1^2(2c_2+1)^2\rho_X h_X
  \]
  and thus also
  \[
  \|M\|_{2,2} \le \|M\|_{\infty,\infty}^{1/2}\|M\|_{1,1}^{1/2} \le c_1^2
  (2c_2+1)^2\rho_X h_X,
  \]
  which yields the stated bound.
\end{proof}

\section{Sampling on Sparse Grids}\label{sec:samponsparse}

In this section, we will derive sampling inequalities on sparse grids,
which are formed by using univariate quasi-uniform point sets. This
generalizes and improves results from \cite{Rieger-Wendland-17-1} and
extends results from \cite{Rieger-Wendland-24-1}.

\subsection{Smolyak's Algorithm and Sparse Grids}
We will now shortly review the necessary material on sparse grids and
Smolyak's algorithm for constructing high-dimensional approximation
operators from one-dimensional operators. To this end, we will use the
following notation.

For $d\in\N$ let $I^d=I\times\cdots\times I \subseteq\R^d$ be
the $d$-fold cube in $\R^d$ and, for $m\in \N$ let $H^m_{\mix}(I^d)$ be
the Sobolev space of mixed regularity of order $m$, i.e.
\begin{eqnarray*}
  H^m_{\mix}(I^d)& = & H^m(I)\otimes\cdots\otimes H^m(I)\\
  & = & \{f\in L_2(I^d) :
D^\aalpha f \in L_2(I^d) \mbox{ for } \aalpha\in\N_0^d \mbox{ with }
\|\aalpha\|_\infty\le m\},\\
\|f\|_{H^m_{\mix}(I^d)}^2 & = & \sum_{\substack{
  \aalpha\in\N_0^d\\ \|\aalpha\|_\infty\le m}}\|D^\aalpha
f\|_{L_2(I^d)}^2.
\end{eqnarray*}

Next, for $\ell\in\N$ let
$X^{(\ell)}=\{x_1^{(\ell)},\ldots,x_{N_\ell}^{(\ell)}\}\subseteq I$ be
given point sets with mesh norms $h_\ell:=h_{X^{(\ell)}}$, separation
radii $q_\ell:=q_{X^{(\ell)}}$ and ratios
$\rho_\ell:=h_\ell/q_\ell$. Assume that we have a locally polynomial
reproducing quasi-interpolation process, $X\mapsto A_X$, which we will
apply to the point sets $X^{(\ell)}$, writing this in the form
\begin{equation}\label{Al}
A^{(\ell)} f (x) = A_{X^{(\ell)}} f(x) = \sum_{j=1}^{N_\ell}
  a_j^{(\ell)}(x) f(x_j^{(\ell)}), \qquad x\in I, f\in C(I).
\end{equation}

We will discuss here only the case of the classic, {\em isotropic}
Smolyak operator, see \cite{Smolyak-63-1,Barthelmann-etal-00-1}. To this end, for $p\in\N$, $p\ge d$, we define the
index sets $Q(p,d)= \{\ii\in\N^d : |\ii|\le p\}$ and $P(p,d) =
\{\ii\in\N^d : p-d+1\le |\ii|\le p\}$. Then, Smolyak's operator for
the index set $Q(p,d)$ has the form
\begin{eqnarray*}
\cA_{Q(p,d)} f (\xx)  &=&  \sum_{\eell\in P(p,d)} (-1)^{p-|\eell|}
\binom{d-1}{p-|\eell|} (A^{(\ell_1)}\otimes\cdots\otimes A^{(\ell_d)})u(\xx)\\
& = & \sum_{\eell\in P(p,d)} (-1)^{p-|\eell|}
\binom{d-1}{p-|\eell|} \sum_{\jj\le \NN}
f\left(x_{j_1}^{(\ell_1)},\ldots x_{j_d}^{(\ell_d)}\right) \prod_{k=1}^da_{j_k}^{(\ell_k)}(x_k).
\end{eqnarray*}
It obviously requires the knowledge of the function $f$ on the sparse
grid
\[
H(p,d)=\bigcup_{\eell\in P(p,d)} \left(X^{(\ell_1)}\times\cdots\times X^{(\ell_d)} \right).
\]

\subsection{Sampling Inequalities}

  In this section, we will derive a general sampling inequality for
  functions $f\in H^m_{\mix}(I^d)$ using, besides the Sobolev norm,
  the $L_2(I^d)$ and $\ell_2(H(p,d))$ norms. The proof is based upon
  the simple triangle estimate
  \begin{equation}\label{gensplit}
\| f\|_{L_2(I^d)} \le \|f-\cA_{Q(p,d)}f \|_{L_2(I^d)} +
\|\cA_{Q(p,d)}f\|_{L_2(I^d)}.
\end{equation}
 We will give bounds for both terms on the right-hand side. Bounds on
 the first term are well-known and follow, for example, from the
 following generic result, see
 \cite[Lemma 2]{Wasilkowski-Wozniakowski-95-1}.

 \begin{lemma}\label{lem:genconvresult}
 Let $U$ and $V$ be normed spaces with tensor products
 $S=U\otimes\cdots \otimes U$ and $T=V\otimes\cdots\otimes V$ .
 Assume that there are constants $C,\widetilde{C}>0$  and $D\in (0,1)$
 such that the  univariate operators 
$A^{(\ell)}: U\to V$ and the embedding operator $E_1:U\to V$ satisfy the  
bounds
\begin{enumerate}
\item  $ \|E_1\|_{U\to V}  =   1$, 
\item  $ \|E_1-A^{(\ell)}\|_{U\to V}  \le  C D^{\ell}$,  for $ \ell
  \ge 0$, 
\item $ \|A^{(\ell)}-A^{(\ell-1)}\|_{U\to V}  \le \widetilde{C}
  D^{\ell}$, for  $\ell \ge 1$.
\end{enumerate}
Then, the error for the Smolyak operator can  be bounded by 
\[
\|\cE_d-\cA_{Q(p,d)}\|_{S\to T} \le  C \left( \max\left\{
\frac{1}{D},\widetilde{C}\right\}
\right)^{d-1}\genfrac{(}{)}{0pt}{}{p}{d-1} D^{p}, \quad \text{for all
} p\ge d,\] 
where $\cE_d=E_1\otimes \cdots\otimes E_1$ is the embedding of $S$
into $T$ and $A^{(0)}(u)=0\in V$ for all $u\in U$.

\end{lemma}
 Note that the third assumption in the theorem above is actually a
 consequence of the second one simply by triangle inequality. 
 We only state it here, as it is usually done so in this context. The first assumption is usually
 easily satisfied such that the only challenge is given by the second
 assumption.
 
 This generic result gives, in our situation the following error
bound.

 \begin{theorem}\label{thm:contSmolyak}
Assume that the Smolyak operator is built using a local polynomial
reproduction process of degree $m-1$. Assume further that the
univariate point sets are quasi-uniform with $N_1=1$ and
$N_\ell=2^{\ell-1}+1$, $\ell\ge 2$, 
points and thus $h_\ell\sim 2^{-\ell+1}$. Let $N$ be the number of points
in the sparse grid $H(p,d)$ with a fixed $p\ge d$. Then, there are
constants $C^{(i)}_d>0$, $i=1,2$ such that
\[
\|\cE_d-\cA_{Q(p,d)} \|_{H^m(I^d)\to L_2(I^d)}\le C^{(1)}_d
p^{d-1}2^{-(p-d) m}
\le C^{(2)}_d (\log N)^{(m+1)(d-1)} N^{-m},
\]
where in particular $C^{(1)}_d = c^d/(d-1)!$ for some $c\in (0,\infty)$.
 \end{theorem}
 \begin{proof}
  We will apply the generic result above with $U=H^m(I)$ and
  $V=L_2(I)$ and thus $S=H_{\mix}^m(I^d)$ and $T=L_2(I^d)$.  
 We obviously have that the embedding operator $E_1:H^m(I)\to L_2(I)$
 satisfies  the first assumption in Lemma
 \ref{lem:genconvresult}. To see that the second assumption is satisfied with
 $D=2^{-m}$, we use that due to the quasi-uniformity of the data sets,
 there is a $\gamma>0$
 that $h_\ell\le \gamma
 2^{-\ell+1}$. Hence,  if we denote the constant from
 \ref{lem:errorbound1d} by $C_{m}$ then we have by said lemma, 
 \[
 \|E_1-A^{(\ell)}\|_{H^m(I)\to L_2(I)} \le
 C_{m} h_\ell^m
 \le C_{m} \gamma^m
 2^{(-\ell+1) m} = C_{m} \gamma^m 2^{m} \left(2^{-m}\right)^{\ell},
 \]
 
 This yields $C=C_{m} (2\gamma)^{m}$ for $C$ from Lemma
 \ref{lem:genconvresult}. As mentioned above, the third assumption follows from 
 the second, i.e  we have 
 \begin{align*} 
 \|A^{(\ell)}-A^{(\ell-1)}\|_{H^m(I)\to L_2(I)} &\le  \|A^{(\ell)}-E_1
   \|_{H^m(I)\to L_2(I)} +\| E_1- A^{(\ell-1)}\|_{H^m(I)\to L_2(I)}\\ 
 &\le C \left(2^{-m}\right)^{\ell}+  C  \left(2^{-m}\right)^{\ell-1} =
   (2^m+1)C\left(2^{-m}\right)^{\ell},  
 \end{align*}
showing $\widetilde{C}=(2^m+1)C= (2^m+1)
C_{m} (2\gamma)^{2m} \le 2^{2m+1}\gamma^m C_{m}$.
 Thus, Lemma \ref{lem:genconvresult} immediately yields for $p\ge d$
 the bound 
 \begin{equation}\label{errorest1}
\|\cE_d-\cA_{Q(p,d)}\|_{H^m_{\mix}(I^d)\to L_2(I^d)} \le \widetilde{C}_d^{(1)}
 \binom{p}{d-1}2^{-mp}
 \end{equation}
with $\widetilde{C}_d^{(1)}:= C^{d}_{m} \gamma^{dm} 2^{(d-1)(2m+1)+m} $.
This can further be bounded by
\[
\|\cE_d-\cA_{Q(p,d)}\|_{H^m_{\mix}(I^d)\to L_2(I^d)} \le \widetilde{C}_d^{(1)}
\frac{p^{d-1}}{(d-1)!} 2^{-mp},
\]

Next, for  the number of points $N$ of the sparse grid built with
non-nested point sets $X^{(\ell)}$ we have  the 
bounds 
\begin{equation}\label{Nbounds_first} 
2^{p-d}\le N=|H(p,d)| \le 2^{p+1}\binom{p-1}{d-1}, 
\end{equation} 
where the upper bound follows from $N_\ell \le 2^\ell-1$ and thus
\cite[Lemma 7]{Wasilkowski-Wozniakowski-95-1} and the lower bound
follows from the fact that one of the indices $\eell\in\N^d$ with
$|\eell|=p$ has the form $\eell=(1,\ldots,1,p-d+1)$ leading to $|H(p,d)|\ge
|X^{(1)}\times\cdots\times X^{(1)}\times X^{(p-d+1)}|= 2^{p-d}+1\ge 2^{p-d}$.
  
which yield on the one hand the bound $p\le \log_2 N+d\le c_d \log N$
with a certain $c_d>0$, which we can choose dimension independent as
$2 /\log 2$ if $N\ge 2^d$,  and on the other hand
\begin{equation}\label{Nbounds}
2^{-pm} \le N^{-m} \binom{p-1}{d-1}^m 2^m.
\end{equation} 
  Inserting both in (\ref{errorest1})
 yields
\begin{eqnarray*}
\|\cE_d-\cA_{Q(p,d)}\|_{H^m_{\mix}(I^d)\to L_2(I^d)} &\le&
\widetilde{C}_d^{(1)} \binom{p-1}{d-1}^{m+1} 2^m N^{-m}
\le  \widetilde{C}_d^{(1)} 2^m \frac{p^{(d-1)(m+1)}}{[(d-1)!]^{m+1}} N^{-m}\\
&\le & \widetilde{C}_d^{(1)} \frac{2^m c_d^{(d-1)(m+1)}}{[(d-1)!]^{m+1}}
(\log N)^{(m+1)(d-1)} N^{-m}, 
\end{eqnarray*}
which is the second stated bound.
   \end{proof}

Next, we turn to the second term on the right-hand side of
(\ref{gensplit}). To derive a bound, we mimic the ideas leading to
Lemma \ref{lem:discrete}. Associated to the operator
$A^{(\ell)}:C(I)\to L_2(I)$ from
(\ref{Al}) is the operator $\widetilde{A}^{(\ell)}:\R^{N_\ell}\to L_2(I)$ defined by
\[
\widetilde{A}^{(\ell)} \cc (x) = \sum_{j=1}^{N_\ell}
  a_j^{(\ell)}(x)c_j, \qquad \cc\in\R^{N_\ell}, \quad x\in I.
  \]
and the sampling operator $S^{(\ell)}:C(I)\to \R^{N_\ell}$ defined by
$S^{(\ell)}f:=(f(x_1^{(\ell)}),\ldots, f(x_{N_\ell}))^\transpose$ such
  that we have again $A^{(\ell)}= \widetilde{A}^{(\ell)}\circ
  S^{(\ell)}$.

  Next, we need to recall, for details see for example
  \cite{Hackbusch-12-1}, that the tensor product of the spaces 
  $(C(I), \|\cdot\|_{L_\infty(I)})$ using the injective norm is
   isomorphic to $(C(I^d), L_\infty(I^d))$ and the tensor
  product of the spaces $L_2(I)$ with the standard $L_2(I)$ norm is
  isomorphic to $(L_2(I^d),\|\cdot\|_{L_2(I^d)})$ and that
  the norms are compatible. Moreover,  the tensor product of
  $(\R^{n_{j}},\|\cdot\|_{\ell_2})$, $1\le j\le d$, is given by
   $(\R^{n_1}\otimes\cdots\otimes \R^{n_d}, \|\cdot\|_{\ell_2})$,
  where  the norm of any tensor $\aa\in \R^{n_1}\otimes\cdots\otimes \R^{n_d}$,
  when written in the standard basis as
  \[
  \aa = \sum_{i_1=1}^{n_1}\cdots \sum_{i_d=1}^{n_d}
  a_{i_1,\ldots,a_{i_d}} \ee_{i_1}^{(n_1)}\otimes\cdots\otimes
    \ee_{i_d}^{(n_d)},
    \]
    is given by
    \[
    \|\aa\|_{\ell_2}^2 = \sum_{i_1=1}^{n_1}\cdots\sum_{i_d=1}^{n_d}
    a_{i_1,\ldots,i_d}^2.
    \]
  Thus,  we have for any  $\eell=(\ell_1,\ldots,\ell_d)^\transpose\in\N^d$
  that the tensor product operators are mappings of the form
  \begin{eqnarray*}
    A^{(\ell_1)}\otimes\cdots\otimes A^{(\ell_d)}&:& (C(I^d),
  \|\cdot\|_{L_\infty(I^d)})\to (L_2(I^d),\|\cdot\|_{L_2(I^d)}),\\
  \widetilde{A}^{(\ell_1)}\otimes\cdots\otimes
  \widetilde{A}^{(\ell_d)}&:&
  (\R^{N_{\ell_1}}\otimes\cdots\otimes\R^{N_{\ell_d}},
  \|\cdot\|_{\ell_2})\to (L_2(I^d),\|\cdot\|_{L_2(I^d)}),\\
  S^{(\ell_1)}\otimes\cdots\otimes S^{(\ell_d)}&:&
  (C(I^d),\|\cdot\|_{L_\infty(I^d)})) \to  (\R^{N_{\ell_1}}\otimes\cdots\otimes\R^{N_{\ell_d}},
  \|\cdot\|_{\ell_2}).
\end{eqnarray*}

\begin{lemma}
  Under the assumptions above, the involved operators satisfy
  \[
  A^{(\ell_1)}\otimes\cdots \otimes A^{(\ell_d)} = (\widetilde{A}^{(\ell_1)}\otimes\cdots\otimes
  \widetilde{A}^{(\ell_d)}) \circ ( S^{(\ell_1)}\otimes\cdots\otimes
  S^{(\ell_d)})
  \]
  for all $\eell\in\N^d$.
\end{lemma}
\begin{proof}
For an elementary tensor $f=f_1\otimes\cdots \otimes f_d\in C(I^d)$  We have 
  the identity
\begin{eqnarray*}
  A^{(\ell_1)}\otimes\cdots \otimes A^{(\ell_d)} (f) & = &
  (\widetilde{A}^{(\ell_1)}\circ S^{(\ell_1)}) \otimes \cdots\otimes
  (\widetilde{A}^{(\ell_d)}\circ S^{(\ell_d)}) (f_1\otimes\cdots
  \otimes f_d)\\
  & = & (\widetilde{A}^{(\ell_1)}\circ S^{(\ell_1)})(f_1)\otimes\cdots
  \otimes (\widetilde{A}^{(\ell_d)}\circ S^{(\ell_d)})(f_d)\\
  & = & \widetilde{A}^{(\ell_1)}(S^{(\ell_1)}f_1)\otimes \cdots
  \otimes \widetilde{A}^{(\ell_d)}(S^{(\ell_d)}f_d)\\
  & = &(\widetilde{A}^{(\ell_1)}\otimes\cdots\otimes
  \widetilde{A}^{(\ell_d)}) (S^{(\ell_1)}f_1\otimes\cdots\otimes S^{(\ell_d)}f_d)\\
  & = & (\widetilde{A}^{(\ell_1)}\otimes\cdots\otimes
  \widetilde{A}^{(\ell_d)}) \circ ( S^{(\ell_1)}\otimes\cdots\otimes
  S^{(\ell_d)}) (f)
\end{eqnarray*}
and the general case follows by linearity and density.
\end{proof}

With this, we have the following result, which also only
requires the second and third property of the local polynomial
reproduction process, i.e. the boundedness of the Lebesgue functions
and the locality of the weights. The polynomial reproduction is not necessary.

\begin{lemma}\label{lem:discrete2}
  Assume that the Smolyak operator is built using a local polynomial
  reproduction process of degree $m-1$. Then,
  \[
   \| A^{(\ell_1)}\otimes\cdots \otimes A^{(\ell_d)}f\|_{L_2(I^d)} 
   \le 
  [c_1(2c_2+1)]^d \prod_{k=1}^d (\rho_{\ell_k}h_{\ell_k})^{1/2}
  \|f\|_{\ell_2(X)}
  \]
  for all $f\in C(I^d)$. 
\end{lemma}
\begin{proof}
  From the above considerations, we have for $f\in C(I^d)$
  with $X=X^{(\ell_1)}\times\cdots\times X^{(\ell_d)}$ the bound
  \begin{eqnarray*}
    \| A^{(\ell_1)}\otimes\cdots \otimes A^{(\ell_d)}f\|_{L_2(I^d)} & = &
    \|(\widetilde{A}^{(\ell_1)}\otimes\cdots\otimes
  \widetilde{A}^{(\ell_d)}) \circ ( S^{(\ell_1)}\otimes\cdots\otimes
  S^{(\ell_d)}) (f)\|_{L_2(I^d)} \\
  & \le &     \|(\widetilde{A}^{(\ell_1)}\otimes\cdots\otimes
  \widetilde{A}^{(\ell_d)})\|_{\ell_2(X)\to L_2(I^d)}
  \|f\|_{\ell_2(X)}\\
  & \le & \left(\prod_{k=1}^d
  \|\widetilde{A}^{(\ell_k)}\|_{\ell_2(X^{(\ell_k)})\to L_2(I)}
  \right)\|f\|_{\ell_2(X)}\\
  & \le & \prod_{k=1}^d
  \left(c_1(2c_2+1)(\rho_{\ell_k}h_{\ell_k})^{1/2}\right)\|f\|_{\ell_2(X)}\\
  & = & [c_1(2c_2+1)]^d \prod_{k=1}^d (\rho_{\ell_k}h_{\ell_k})^{1/2} \|f\|_{\ell_2(X)},
  \end{eqnarray*}
  which is the stated bound.
  \end{proof}

This result now yields the required bound on the second term on the
right-hand side of (\ref{gensplit}), noting also that we have a
continuous embedding $H^m_{\mix}(I^d)\subseteq C(I^d)$ for all $m\in\N$.

\begin{theorem}\label{thm:discreteSmolyak}
Assume that the Smolyak operator is built using a local polynomial
reproduction process of degree $m-1$. Assume further that the
univariate point sets are quasi-uniform with $N_1=1$ and
$N_\ell=2^{\ell-1}+1$, $\ell\ge 2$, points and thus $h_\ell\sim
2^{-\ell+1}$. Let $N$ be the number of points 
in the sparse grid $H(p,d)$ with a fixed $p\ge d$. Then, there are
constants $C^{(i)}_d>0$, $i=3,4$ such that
\begin{eqnarray*}
  \|\cA_{Q(p,d)}f \|_{L_2(I^d)} &\le& C^{(3)}_d p^{d-1}2^{(d-p)/2} \|f\|_{\ell_2(H(p,d))}\\
& \le &    C^{(4)}_d (\log N)^{3(d-1)/2} N^{-1/2} \|f\|_{\ell_2(H(p,d))}
\end{eqnarray*}
for all $f\in H^m_{\mix}(I^d)$ and $C^{(3)}_d=\frac{c^d}{(d-1)!}$ for some $c\in (0,\infty)$.
\end{theorem}
\begin{proof}
The quasi-uniformity assumption on the data sites shows that there is
a $\rho>0$ such that $\rho_\ell\le \rho$ for all
$\ell\in\N$. Moreover, there is a $\gamma>0$ such that $h_\ell\le
\gamma 2^{-\ell+1}$ for all $\ell\in\N$.

Lemma \ref{lem:discrete2} then yields
\begin{eqnarray*}
\|\cA_{Q(p,d)}f\|_{L_2(I^d)} & \le & \sum_{\eell\in P(p,k) }
  \binom{d-1}{p-|\eell|} \|(A^{(\ell_1)}\otimes\cdots\otimes A^{(\ell_d)})f\|_{L_2(I^d)}\\
  & = & \sum_{j=p-d+1}^p
  \binom{d-1}{p-j}\sum_{|\eell|=j }
  \|(A^{(\ell_1)}\otimes\cdots\otimes A^{(\ell_d)})f\|_{L_2(I^d)}\\
  & \le & \sum_{j=p-d+1}^p
  \binom{d-1}{p-j}\sum_{|\eell|=j } c_1^d(2c_2+1)^d\prod_{k=1}^d
  (\rho_{\ell_k}h_{\ell_k})^{1/2} \|f\|_{\ell_2(H(p,d))}\\
  &\le& [c_1(2c_2+1)(\rho\gamma)^{1/2}]^d \sum_{j=p-d+1}^p
  \binom{d-1}{p-j}\sum_{|\eell|=j } \prod_{k=1}^d 2^{-(\ell_k-1)/2}
  \|f\|_{\ell_2(H(p,d))}\\
  & = & [2^{1/2}c_1(2c_2+1)(\rho\gamma)^{1/2}]^d \sum_{j=p-d+1}^p
  \binom{d-1}{p-j}\binom{j-1}{d-1}  2^{-j/2}
  \|f\|_{\ell_2(H(p,d))},
\end{eqnarray*}
where we have used that the cardinality of $\{\eell\in\N^d :
|\eell|=j\}$ is given by $\binom{j-1}{d-1}$. The monotonicity of the
latter binomial coefficient allows us to conclude
\begin{eqnarray*}
  \sum_{j=p-d+1}^p 
  \binom{d-1}{p-j}\binom{j-1}{d-1}  2^{-j/2} & \le& \binom{p-1}{d-1} \sum_{j=p-d+1}^p
  \binom{d-1}{p-j}  2^{-j/2}\\
    & = & \binom{p-1}{d-1} \sum_{k=0}^{d-1} \binom{d-1}{k}
  2^{(k-p)/2}\\ 
    & =& \binom{p-1}{d-1}2^{-p/2} \sum_{k=0}^{d-1}\binom{d-1}{k} 2^{k/2} \\
  & = & \binom{p-1}{d-1}2^{-p/2} \left(1+2^{1/2}\right)^{d-1}.
\end{eqnarray*}
Plugging this into the above estimate, we obtain
\begin{eqnarray*}
\|\cA_{Q(p,d)}f \|_{L_2(I^d)} &\le&
     [(2^{1/2}+2) c_1 (2c_2+1)(\rho\gamma)^{1/2}]^d
     \binom{p-1}{d-1} 2^{-p/2}  \|f\|_{\ell_2(H(p,d))} \\
     &=:& \widetilde{C}_d^{(3)}     \binom{p-1}{d-1} 2^{-p/2}  \|f\|_{\ell_2(H(p,d))} \\
\end{eqnarray*}
for all $f\in H^m_{\mix}(I^d)$. Using again $\binom{p-1}{d-1}\le
p^{d-1}/(d-1)!$ shows the first stated bound with $C_d^{(3)} =
\widetilde{C}_d^{(3)}2^{-d/2}/(d-1)!$. 
For the second bound, we employ again  $p\le C_d\log N$ and
(\ref{Nbounds}) with $m=1/2$. This leads to

\begin{eqnarray*}
 \|\cA_{Q(p,d)}f\|_{L_2(I^d)}
  & \le & \widetilde{C}_d^{(3)} \binom{p-1}{d-1} 2^{-p/2}
 \|f\|_{\ell_2(H(p,d))}
 \le \widetilde{C_d}^{(3)} \binom{p-1}{d-1}^{3/2} N^{-1/2} 2^{1/2}\\
  &\le &
  \frac{\widetilde{C}_d^{(3)}2^{1/2}}{[(d-1)!]^{3/2}}  p^{3(d-1)/2}N^{-1/2} 
  \|f\|_{\ell_2(H(p,d))}\\
  &\le &
    \frac{\widetilde{C}_d^{(3)}2^{1/2}C_d^{3(d-1)/2}}{[(d-1)!]^{3/2}}
    (\log N)^{3(d-1)/2}N^{-1/2} 
  \|f\|_{\ell_2(H(p,d))}\\
\end{eqnarray*}
which is the desired second bound and defines the constant $C_d^{(4)}$.
\end{proof}

Returning now to our initial goal, 
we can plug the results of the previous two theorems into
(\ref{gensplit}) to derive the following sampling inequality.

\begin{theorem}\label{thm:error1}
Assume that the Smolyak operator is built using a local polynomial
reproduction process of degree $m-1$. Assume further that the
univariate point sets are quasi-uniform with $N_\ell=2^{\ell-1}+1$
points and thus $h_\ell\sim 2^{-i+1}$. Let $N$ be the number of points
in the sparse grid $H(p,d)$ with a fixed $p\ge d$. Then, there is a
constant $C_d>0$ such that
\begin{eqnarray*}
  \|f\|_{L_2(I^d)} &\le&
 p^{d-1}\left[C_d^{(1)} 2^{-(p-d)m}\|f\|_{H^m_{\mix}(I^d)} + C_d^{(3)}
   2^{-(p-d)/2} \|f\|_{\ell_2(H(p,d))}\right]\\ 
  &\le & C_d^{(2)}(\log N)^{(m+1)(d-1)} N^{-m}
\|f\|_{H^m_{\mix}(I^d)} + C_d^{(4)} (\log N)^{3(d-1)/2} N^{-1/2}
\|f\|_{\ell_2(H(p,d))}
\end{eqnarray*}
for all $f\in H^m_{\mix}(I^d)$. Note that both $C_d^{(1)}$ and
$C_d^{(3)}$ are of the form $c^d/(d-1)!$.
\end{theorem}

\section{Sampling of High-dimensional Functions with Low Effective Dimension}
\label{sec:samphighwithlow}

Given a large space dimension $d\in\N$, we now want to apply the
results of the previous section to specific functions from
$H_{\mix}^m(I^d)$, which are effectively low dimensional, i.e. they
are the sum of functions that depend on fewer than $d$ variables.

To this end, we have to recall the following definitions and
results. Details can be found, for example, in
\cite{Rieger-Wendland-24-1}.

Let $\cP(\cD)$ be the power set of $\cD=\{1,\ldots,d\}$, i.e. the set
of all subsets of $\cD$. A set of subsets $\Lambda\subseteq \cP(\cD)$
is called {\em downward closed} if for any $\mfu\in \Lambda$ and any
$\mfv\in\cP(\cD)$ with $\mfv\subseteq\mfu$ also $\mfv\in \Lambda$ is
given. The complement of $\Lambda$ in $\cP(\cD)$ is denoted by $\complement\Lambda$.

For any $\xx\in I^d$ and $\mfu\in \cD$, we denote the number of
elements in $\mfu$ by $|\mfu|$, define 
$I_{\mfu}:=\bigtimes_{j\in \mfu} I =  I^{|\mfu|}$ and $\xx_\mfu$ as
the vector in $I^{|\mfu|}$ having the entries $(x_j : j\in
\mfu)$. Finally, for a fixed $\cc\in I^d$, we let
$(\xx;\cc)_\mfu\in I^d$ be the point having entries $x_j$ for $j\in \mfu$
and $c_j$ for $j\in\cD\setminus\mfu$. In the same way, we might start
with an $\widetilde{\xx}\in I^{|\mfu|}$ and extend it to a point
$\xx\in I^d$ by setting $x_j=\widetilde{x}_j$ for $j\in \mfu$ and
$x_j=c_j$ for $j\in\cD\setminus \mfu$, which we will also denote with
$\xx=(\widetilde{\xx};\cc)_\mfu$. We will derive sampling inequalities
for point sets on low-dimensional sparse grids.

\begin{definition}\label{def:samplingset}
Let $N_1=1$ and $N_\ell=2^{\ell-1}+1$  for $\ell \ge 2$. For $\ell\in\N$ let
$Y^{(\ell)}\subseteq I$,  be a sequence of  quasi-uniform,  point sets
in $I$ with $|Y^{(\ell)}|=N_\ell$ points. 
For $\emptyset\ne\mfu\subseteq\cD$ and $p\in\N$ with $p\ge |\mfu|$, the {\em
  sparse grid} $\widetilde{X}_{\mfu,p}\subseteq I_\mfu$ based on
the points $\{Y^{(\ell)}\}$ is defined as
\[
\widetilde{X}_{\mfu,p} = \bigcup_{\eell\in P(p,|\mfu|)} 
Y^{(\ell_1)}\times \cdots \times Y^{(\ell_{|\mfu|})}.
\]
In the case of $\mfu=\emptyset$ we set $\widetilde{X}_{\emptyset,p} =
Y^{(p)}$. The points of these low-dimensional sets are extended using the anchor
  $\00\in I^d$, yielding  point sets $X_{\mfu,p}\subseteq I^d$.

For $\pp=\{p_\mfu : \mfu\in \Lambda\}$ with $p_\mfu \ge |\mfu|$,
  a {\em sampling point set for mixed regularity Sobolev 
  $\Lambda$-functions} is given by
  \begin{equation}\label{mixedpoints}
X_{\Lambda,\pp}^{(d)}=\bigcup_{\mfu\in\Lambda} X_{\mfu,p_\mfu}.
\end{equation}
\end{definition}

Note that in contrast to our results in \cite{Rieger-Wendland-24-1},
we do not assume that our univariate point sets $Y^{\ell}$ are
nested. Thus, we need to modify the estimates on the number of points
in $X_{\Lambda,\pp}^{(d)}$ slightly and cannot use the results from
\cite[Proposition 5.5]{Rieger-Wendland-24-1} directly.

\begin{lemma}\label{lem:N}
 Let $\Lambda\subseteq\cP(\cD)$ a downward closed set
 with $n:=\max\{|\mfu| : \mfu\in\Lambda\}\le d$.
   Let $p\in\N$ be fixed. If   $p_{\mfu} = p + |\mfu|$
 for every $\mfu\in\Lambda\}$ then the number of points in
 $X_{\Lambda,\pp}^{(d)}$ can be bounded by
 \begin{equation}\label{eq:boundXLambda}
 2^{p-n}|\Lambda| \le |X_{\Lambda,\pp}^{(d)}| \le 2^{p+n+1}p^{n-1}(1+n)^{n-1} |\Lambda|.
 \end{equation}
In particular, if $\Lambda=\Lambda_n=\{\mfu\subseteq\{1,\ldots,d\} :
|\mfu|\le n\}$ this becomes
 \begin{equation}\label{eq:boundXLambdappd}
2^{p-n}\left(\frac{d}{n}\right)^n\le  |X_{\Lambda,\pp}^{(d)}| \le 2^{p+1}p^{n-1} (4ed)^n.
 \end{equation}
 
\end{lemma}

\begin{proof}
  We start with \eqref{Nbounds_first} in the form 
\[
 |X_{\mfu,p_\mfu}| \le 2^{p+|\mfu|+1} \binom{p+|\mfu|-1}{|\mfu|-1},
\qquad \mfu \ne \emptyset.
\]
to derive
\begin{eqnarray*}
  |X_{\Lambda,\pp}^{(d)}| &\le& \sum_{\emptyset\ne\mfu \in \Lambda}
  2^{p+|\mfu|+1} \binom{p+|\mfu|-1}{|\mfu|-1} + N_p \\
  & \le & 2^{p+n+1} (p+n)^{n-1} (|\Lambda|-1) + 2^{p-1}+1\\
  & \le & 2^{p+n+1} p^{n-1}(1+n)^{n-1} |\Lambda|,
\end{eqnarray*}
where we have used that $(p+|\mfu|)^{|\mfu|-1} \le (p+n)^{n-1} \le
p^{n-1}(1+n)^{n-1}$, which can easily be seen, using the binomial
expansion of $(p+n)^{n-1}$. This gives our general upper bound from
(\ref{eq:boundXLambda}). The 
lower bound follows in the same way from 
\[
|X_{\Lambda,\pp}^{(d)}| \ge \sum_{\emptyset\ne\mfu\in\Lambda}
2^{p-|\mfu|} + N_p \ge 2^{p-n} |\Lambda|.
\]
For the specific bound for $\Lambda_n$, we can use \cite[Lemma
  2.9]{Rieger-Wendland-24-1}, which states that the cardinality 
  of $\Lambda$ can be bounded by
\[
\left(\frac{d}{n}\right)^n\le  |\Lambda| = \sum_{\ell=0}^n \binom{d}{\ell} \le
\left(\frac{ed}{n}\right)^n.
\]
This yields for the upper bound
\begin{eqnarray*}
|X_{\Lambda,\pp}^{(d)}| &\le&  2^{p+n+1} p^{n-1}(1+n)^{n-1}
\left(\frac{ed}{n}\right)^n
 \le  2^{p+1} p^{n-1} \frac{1}{1+n}
 \left(\frac{2ed(n+1)}{n}\right)^n\\
 & \le & 2^{p+1}p^{n-1}(4ed)^n,
  \end{eqnarray*}
which is the stated upper bound. The lower bound follows similarly.
\end{proof}

Besides the sampling points just introduced, we need a subspace of the
mixed regularity Sobolev functions, i.e. such functions with an
effective low dimension. For a downward closed set  $\Lambda$, this
subspace is defined as follows.

\begin{definition}\label{def:hmmix}
For a downward closed set $\Lambda\subseteq\cP(\cD)$, the space
$H^m_{\mix,\Lambda}(I^d)$ consists of all function $f\in
H^m_{\mix}(I^d)$ having a representation of the form
\[
f(\xx) = \sum_{\mfu \in \Lambda} f_\mfu (\xx_\mfu), \qquad \xx\in I^d,
\]
i.e. $f$ is the sum of lower dimensional functions $f_\mfu$. 
\end{definition}

To understand these spaces better, we recall some
additional, required facts on mixed regularity Sobolev spaces. We can equip
 $H^m(I)$ with the inner product
\[
\langle f, g\rangle_{H^m([0,1])} = \sum_{r=0}^{m-1} f^{(r)}(0)
  g^{(r)}(0) + \int_0^1 f^{(m)}(y) g^{(m)}(y) d y,
  \]
which yields a norm which is equivalent to the standard norm. With
this inner product, $H^m(I)$ is a reproducing kernel Hilbert space
with kernel
\begin{equation}\label{littlek}
  k_m(x,y) =\sum_{r=0}^{m-1}\frac{x^r}{r!}\frac{y^r}{r!} + \int_0^1
  \frac{(x-z)_+^{m-1}}{(m-1)!} \frac{(y-z)_+^{m-1}}{(m-1)!} dz.
\end{equation}
Then standard tensor product theory yields that $H^m_{\mix}(I^d)$ is a
reproducing kernel Hilbert space. Its reproducing kernel can be written
as
\[ 
K_m(\xx,\yy)=\prod_{j=1}^d k_m(x_j,y_j) = \sum_{\mfu\subseteq \cD} K_\mfu(\xx,\yy), \qquad
  K_\mfu(\xx,\yy) = \prod_{j\in\mfu} (k_m(x_j,y_j)-1),
\]
where the product over the empty index set is defined to be $1$. As a
matter of fact, each
$K_{\mfu}$ is the reproducing kernel of some Hilbert space $H_\mfu$
and since $K_m$ satisfies the {\em annihilation property} (see
\cite{Kuo-etal-10-1}), we have the orthogonal decomposition
\[
H_{\mix}^m ([0,1]^d) = \sum_{\mfu\subseteq \cD} H_\mfu.
\]
As we assume $\Lambda$ to be downward closed, the general theory on
projection methods, see again \cite{Kuo-etal-10-1} for details, guarantees that we can
use the anchored components
\begin{equation}\label{fmfu}
f_\mfu(\xx_\mfu) = \sum_{\mfv\subseteq\mfu} (-1)^{|\mfu|-|\mfv|}
f((\xx;\cc)_\mfv)
\end{equation}
in Definition \ref{def:hmmix}.
Using the above inner product and kernel, the next proposition summarizes
further  necessary results on such 
spaces, which are a consequence of \cite[Theorem 3.10]{Rieger-Wendland-24-1}
and \cite[Proposition 6.1]{Rieger-Wendland-24-1}.

\begin{proposition}
 Let $\Lambda\subseteq\cP(\cD)$ be a downward closed set. Then, the
 space $H^m_{\mix,\Lambda}(I^d)$ is a closed subspace of
 $H^m_{\mix}(I^d)$. The space $H^m_{\mix}(I^d)$ has the orthogonal
 decomposition
 \[
 H^m_{\mix}(I^d) = H^m_{\mix,\Lambda}(I^d) \oplus
 H^m_{\mix,\complement\Lambda}(I^d)
 \]
 and the orthogonal projection $P_\Lambda:H^m_{\mix}(I^d)\to
 H^m_{\mix,\Lambda}(I^d)$ is given by
 \[
 P_\Lambda f(\xx) = \sum_{\mfu\in\Lambda}f_\mfu(\xx_\mfu), \qquad
 \xx\in I^d,
 \]
 with $f_\mfu$ from (\ref{fmfu}).
\end{proposition}
After this, we are able to state and prove the main result of this
section, a sampling inequality for $H_{\mix,\Lambda}^m(I^d)$ for data
sets as defined in Definition \ref{def:samplingset}.

\begin{theorem}\label{thm:sampLambda}
  Let $\Lambda\subseteq\cP(\cD)$ be a downward closed set with $n:=\max\{|\mfu| : \mfu\in
  \Lambda\}\le n$. For a fixed $p\in\N$ let $\pp=\{p_{\mfu}=p+|\mfu| \in\N
  :\mfu\in\Lambda\}$  and let  $X_{\Lambda,\pp}^{(d)}$ be the sampling
  data set from Definition \ref{def:samplingset} with
  $N=|X_{\Lambda,\pp}^{(d)}|$ points.
  Then, there are constant $\widetilde{C_1}>0$ and
  $\widetilde{C}_{m,n}^{(1)}, \widetilde{C}_{m,n}^{(2)}>0$ such
  that for all $f\in H^m_{\mix,\Lambda}(I^d)$, 
\begin{eqnarray*}
\|f\|_{L_2(I^d)}  &\le& 
(\widetilde{C_1}d)^n p^{n-1} \left[2^{-pm}\|f\|_{H^m_{\mix}(I^d)} +
  2^{-p/2} \|f\|_{\ell_2(X_{\Lambda,\pp}^{(d)})} \right]\\
& \le& 
    \widetilde{C}_{m,n}^{(1)} d^{(m+1)n} (\log N)^{(m+1)(n-1)} N^{-m}
    \|f\|_{H^m_{\mix}(I^d)} \\
    & & \mbox{}
    + \widetilde{C}_{m,n}^{(2)} d^{3n/2} (\log N)^{\frac{3}{2}(n-1)} N^{-1/2}
        \|f\|_{\ell_2(X_{\Lambda,\pp}^{(d)})}.
\end{eqnarray*}
\end{theorem}
\begin{proof}
 Let $c_I:= \vol(I)^n$.   For any $\mfv\in \Lambda$ we will shortly
 write $X_{\mfv}$ for $X_{\mfv,p_{\mfv}}$. Then,   Theorem
 \ref{thm:error1}  yields for $\mfv\ne\emptyset$ with $|\mfv|\le n$,
  \begin{eqnarray*}
    \|f((\cdot;\cc)_\mfv)\|_{L_2(I^d)} & \le & c_{I}
    \|f((\cdot;\cc)_{\mfv})\|_{L_2(I_\mfv)} \\
    &\le &
     c_{I} p_\mfv^{|\mfv| -1}\left [ C_{|\mfv|}^{(1)} 2^{-(p_{\mfv}-|\mfv|)m}
      \|f((\cdot;\cc)_\mfv)\|_{H^m_{\mix}(I_\mfv)} + C_{|\mfv|}^{(3)} 2^{-(p_\mfv
        -|\mfv|)/2}\|f\|_{\ell_2(X_\mfv)}\right]\\
    & = & c_{I}  c^{|\mfv|} \frac{(|\mfv|+p)^{|\mfv|-1}}{(|\mfv|-1)!}
    \left[  2^{-pm}\|f((\cdot;\cc)_{\mfv})\|_{H^m_{\mix}(I_\mfv)} +
      2^{-p/2} \|f\|_{\ell_2(X_\mfv)}\right]\\
    &\le & c_I c^{n} (n+p)^{n-1}\left[ 2^{-pm}\|f((\cdot;\cc)_{\mfv})\|_{H^m_{\mix}(I_\mfv)} +
      2^{-p/2} \|f\|_{\ell_2(X_\mfv)}\right], 
  \end{eqnarray*}
  where we have used that the constants $C_{|\mfv|}^{(i)}$, $i=1,3$
  can be written in the form $c_i^{|\mfv|}/(|\mfv|-1)!$ and  have set
  $c=\max\{c_1,c_3\}$. Noting that a similar estimate holds in the
  case of $\mfv=\emptyset$, this gives
  \begin{eqnarray*}
\|f\|_{L_2(I^d)} & \le & c_{I} \sum_{\mfu\in
  \Lambda}\sum_{\mfv\subseteq\mfu} 
\|f((\cdot;\cc)_{\mfv})\|_{L_2(I_\mfv)} \\
&\le & c_I c^{n}
(n+p)^{n-1}\sum_{\mfu\in\Lambda}\sum_{\mfv\subseteq\mfu}
\left[ 2^{-pm}\|f((\cdot;\cc)_{\mfv})\|_{H^m_{\mix}(I_\mfv)} +
  2^{-p/2} \|f\|_{\ell_2(X_\mfv)}\right]\\
&\le & c_I c^{n}
(n+p)^{n-1}
\left[ 2^{-pm}\|f\|_{H^m_{\mix}(I^d)} +
      2^{-p/2} \|f\|_{\ell_2(X)}\right]
\sum_{\mfu\in\Lambda}\sum_{\mfv\subseteq\mfu} 1.
%
  \end{eqnarray*}
  Using again
  \[
  \sum_{\mfu\in\Lambda}\sum_{\mfv\subseteq\mfu} 1 = \sum_{\ell=0}^n
  \sum_{|\mfu|=\ell} \sum _{\mfv\subseteq \mfu} 1 = 
  \sum_{\ell=0}^n \binom{d}{\ell}2^\ell \le 2^n
  \left(\frac{ed}{n}\right)^n
  \]
  finally leads to
  \begin{eqnarray*}
    \|f\|_{L_2(I^d)} &\le& c_Ic^n (n+p)^{n-1} 2^n
    \left(\frac{ed}{n}\right)^n 
\left[ 2^{-pm}\|f\|_{H^m_{\mix}(I^d)} +
      2^{-p/2} \|f\|_{\ell_2(X)}\right]\\
&\le&    \left(4\vol(I)ced)\right)^{n} p^{n-1}
  \left[2^{-pm}\|f\|_{H^m_{\mix}(I^d)} + 
    2^{-p/2}\|f\|_{\ell_2(X)}\right],
  \end{eqnarray*}
  where in the last step we used estimates as in the proof of Lemma
  \ref{lem:N}. This shows the first estimate with $\widetilde{c}_n=4\vol(I)ce$.

  Next, by the definition of $n$, we have $\Lambda\subseteq\Lambda_n$,
  thus we can use the upper bound from
  (\ref{eq:boundXLambdappd}) for our $\Lambda$. For the lower bound, we use
  that $|\Lambda|\ge 1$. This means we have altogether
  \[
  2^{p-n}  \le N \le 2^{p+1}p^{n-1} (4ed)^n.
  \]
  
  This leads on the one hand to $p\le c_n \log N$, where we can choose
  $c_n=2/\log 2$ again if $N\ge 2^n$ which is a reasonable assumption
  in this situation, and 
  \[
  2^{-p/2} \le 2^{1/2}(4e)^{n/2} d^{n/2} p^{(n-1)/2} N^{-1/2} \le
  2^{1/2}(4e)^{n/2} c_n^{(n-1)/2} 
  (\log N)^{(n-1)/2}
  N^{-1/2}
  \]
  and on the other hand to
  \[
  2^{-pm} \le 2^m p^{m(n-1)} (4ed)^{mn} N^{-m} \le 2^mc_n^{m(n-1)}
  (4ed)^{mn} (\log N)^{m(n+1)}N^{-m} 
  \]
  Inserting this in the above bound leads to the stated estimate with
  \begin{eqnarray*}
    \widetilde{C}_{m,n}^{(1)} & = & \widetilde{C}_1^n c_n^{(m+1)(n-1)},
    2^m (4e)^{mn} \\
    \widetilde{C}_{m,n}^{(2)} & = &\widetilde{C}_1^n
    c_n^{3(n-1)/2}2^{1/2} (4e)^{n/2}.
  \end{eqnarray*}
\end{proof}

\section{Approximation of Mismeasured functions}

In this final section, we want to apply the results of the previous section
to the following situation. Assume we have a target function $f\in
H_{\mix}^m(I^d)$ which can be decomposed into $f=f_\Lambda +
f_{\complement\Lambda}$ with a dominant part $f_\Lambda=P_\Lambda f
\in H^m_{\mix,\Lambda}(I^d)$. Then, the goal is to approximately reconstruct such an
$f$ from its measurements on a sampling data set $X_{\Lambda,\pp}^{(d)}$
using a representation from $H^m_{\mix,\Lambda}(I^d)$. This is a
continuation of the discussion we started in
\cite{Rieger-Wendland-26-1} and provides a significant improvement
over the results achieved in that paper.

The reconstruction process that we will employ is a penalized
least-squares approach, see \cite{Wahba-90-1}. However, in contrast to
the standard approach, we do not minimize the penalized least-squares
functional over the entire function space but only over a closed
subspace. To be more precise, in general terms, let $H$ be a Hilbert space of functions
$f:\Omega\subseteq\R^d\to\R$  and let $H_1$ be a closed subspace of $H$.  For a data set
$X=\{\xx_1,\ldots,\xx_N\}\subseteq \Omega$, observations
$f(X)=(f(\xx_1),\ldots, f(\xx_N))^\transpose\in \R^N$ of an unknown
function $f\in H$, and a regularization parameter $\lambda>0$, we define
\begin{eqnarray*}
  J_{X,\lambda,f(X)}(s)&:=& \sum_{j=1}^N
  |f(\xx_j)-s(\xx_j)|^2 + \lambda \|s\|_H^2, \qquad s\in H,\\
  Q_{X,\lambda,H_1}(f)&:=&\argmin_{s\in H_1}
  J_{X,\lambda,f(X)}(s).
\end{eqnarray*}
In this general situation, the following result was proven in
\cite{Rieger-Wendland-26-1}. 

\begin{theorem}
  Let $H$ be a reproducing kernel Hilbert space of functions
$f:\Omega\to\R$ with reproducing kernel $K$. Let $H_1$ be a closed
subspace with reproducing kernel
$K_1$ and let $H_2$ be the orthogonal complement to $H_1$ in $H$,
i.e. $H=H_1\oplus H_2$. Assume that there is
a constant $C>0$ and two  functions $F_1,F_2:\N\to [0, \infty)$  such that on $H_1$ a {\em
sampling inequality} of the form 
\begin{equation}\label{genericsamplinginequality}
\|f_1\|_{L_2(\Omega)} \le C \left[ F_1(N) \|f_1\|_H + F_2(N)\|f_1\|_{\ell_2(X)}\right], \qquad
f_1\in H_1, 
\end{equation}
holds. Then, the error between any $f=f_1+f_2\in H$ with $f_i\in H_i$,
$i=1,2$ and $Q_{X,\lambda,H_1}(f)$ can be bounded by
\begin{eqnarray*}
\|f-Q_{X,\lambda,H_1} f\|_{L_2(\Omega)} &\le&
\|f_2\|_{L_2(\Omega)} + C
\left[2F_1(N)+ F_2(N)\sqrt{\lambda}\right]\|f\|_H\\
& & \mbox{} +
  C \left(\frac{F_1(N)}{\sqrt{\lambda}} +2F_2(N)\right) \|f_2\|_{\ell_2(X)}.
\end{eqnarray*}
The particular choice $\sqrt{\lambda}=F_1(N)/F_2(N)$ of the
regularization parameter thus gives the bound
 \[
  \|f-Q_{X,\lambda,H_1}(f)\|_{L_2(\Omega)} \le \|f_2\|_{L_2(\Omega)} + 
  3C \left[F_1(N) \|f\|_H + F_2(N) \|f_2\|_{\ell_2(X)}\right].
  \]
\end{theorem}

Applying this theorem to the situation of the last section means that
we now choose $\Omega=I^d$, $H=H_{\mix}^m(I^d)$, equipped with the
inner product outlined in the last section, and
$H_1=H^m_{\mix,\Lambda}(I^d)$ with a  downward closed
$\Lambda\subseteq\cP(\cD)$ with $n=\max_{\mfu\in\Lambda}|\mfu|$.
Then, Theorem \ref{thm:sampLambda} shows that we have
\begin{eqnarray*}
  F_1(N) & = & (\log N)^{\rho_1(m,n)} N^{-m}, \qquad \rho_1(m,n) =
  (m+1)(n-1),\\
  F_2(N) & = & (\log N)^{\rho_2(n)}N^{-1/2}, \qquad \rho_2(n) =
  \frac{3}{2}(n-1),
\end{eqnarray*}
yielding the following result, where we also use the notation
$Q_{X,\lambda,\Lambda}$ instead of $Q_{X,\lambda,H_1}$ to better
emphasize the subspace.

\begin{corollary}\label{cor:mismatch1}
Let $\Lambda\subseteq\cP(\cD)$ be a downward closed set with
$n=\max_{\mfu\in\Lambda} |\mfu|$. For a fixed $p\in\N$ let
$\pp=\{p_{\mfu}=p+|\mfu| \in\N
:\mfu\in\Lambda \}$  and let
$X=X_{\Lambda,\pp}^{(d)}$ be the sampling data set from Definition \ref{def:samplingset}.
Then, there is a constant $C>0$, depending on $d$ and $m\in\N$, such
that for all $f=f_1+f_2\in H^m_{\mix}(I^d)$ with $f_1\in
H^m_{\mix,\Lambda}(I^d)$ and $f_2\in
H^m_{\mix,\complement\Lambda}(I^d)$ the bound
\begin{eqnarray*}
  \|f-Q_{X,\lambda,\Lambda}(f)\|_{L_2(I^d)} & \le &  C \left[(\log
    N)^{\rho_1(m,n)} N^{-m} + (\log N)^{\rho_2(n)} N^{-1/2}
    \sqrt{\lambda}\right] \|f\|_{H^m_{\mix}(I^d)} \\
  &  &  \mbox{} + C\left[ (\log N)^{\rho_1(m,n)}
    N^{-m}\frac{1}{\sqrt{\lambda}} + (\log N)^{\rho_2(n)}
    N^{-1/2}\right]\|f_2\|_{\ell_2(X_{\Lambda,\pp}^{(d)})} \\
  & & \mbox{} + \|f_2\|_{L_2(I^d)} 
\end{eqnarray*}
holds. The particular choice $\sqrt{\lambda}=F_1(N)/F_2(N)$ yields
\begin{eqnarray*}
\|f-Q_{X,\lambda,\Lambda}(f)\|_{L_2(I^d)} &\le&
C(\log N)^{\rho_1(m,n)} N^{-m} \|f\|_{H^m_{\mix}(I^d)}\\
&& \mbox{}  + C (\log
  N)^{\rho_2(n)}
  N^{-1/2}\|f_2\|_{\ell_2(X_{\Lambda,\pp}^{(d)})} +  \|f_2\|_{L_2(I^d)}\\
& \le & C\left[(\log N)^{\rho_1(m,n)} N^{-m} \|f\|_{H^m_{\mix}(I^d)} +
  (\log N)^{\rho_2(n)} \|f_2\|_{L_\infty(I^d)}\right].
\end{eqnarray*}
\end{corollary}

We will end this section with a short discussion on the requirements
on $f_2$ of $f$ so that we can successfully employ this theory. To
this end, it is more informative to discuss everything in terms of
the relevant parameters $p$ and $n$, rather than $N$. Thus, we rewrite
the error estimate in the last corollary using the first estimate from
Theorem \ref{thm:sampLambda} as the assumed sampling inequality, i.e.
\[
\|f\|_{L_2(I^d)}  \le 
(\widetilde{C_1}d)^n p^{n-1} \left[2^{-pm}\|f\|_{H^m_{\mix}(I^d)} +
  2^{-p/2} \|f\|_{\ell_2(X_{\Lambda,\pp}^{(d)})} \right], \qquad f\in H^m_{\mix,\Lambda_n}(I^d).
\]
Then, with the appropriate choice of $\lambda$, the 
considerations above lead to an error estimate of the form  
\[
\|f-Q_{X,\lambda,\Lambda_n} f\|_{L_2(I^d)}  \le  \|f_2\|_{L_2(I^d)} + 3 (\widetilde{C}_1d)^n
p^{n-1} \left[ 2^{-pm}\|f\|_{H^m_{\mix}(I^d)} + 2^{-p/2}
\|f_2\|_{\ell_2(X_{\Lambda,\pp}^{(d)})}\right].
\]
Using Lemma \ref{lem:N}, assuming for simplicity that $\vol(I)=1$ and
bounding the norms on $f_2$ uniformly leads to 
\begin{eqnarray*}
  \lefteqn{  \|f-Q_{X,\lambda,\Lambda_n}f\|_{L_2(I^d)}}\\
  & \le&
\|f_2\|_{L_\infty(I^d)} + 3 (\widetilde{C}_1d)^n
p^{n-1}  2^{-pm}\|f\|_{H^m_{\mix}(I^d)} 
 +  3 (\widetilde{C}_1d)^n p^{n-1}   2^{p/2+1}p^{n-1} (4ed)^n
\|f_2\|_{L_\infty(I^d)}\\
&\le & 3(\widetilde{C}_1 d)^n p^{n-1}2^{-pm} \|f\|_{H^m_{\mix}(I^d)} +
\left[1+3(4\widetilde{C}_1 ed^2)^np^{2n-2} 2^{p+1}\right]\|f_2\|_{L_\infty(I^d)}.
\end{eqnarray*}

If we now, for the time being, ignore the term $\|f\|_{H^m_{\mix}(I^d)}$ and  want to
bound the first term in the last expression by a given $\epsilon>0$,
i.e. if we want to have $3(\widetilde{C}_1 d)^n
p^{n-1}2^{-pm}<\epsilon$ then we see that we need
\[
p\ge \frac{(n-1) \log (\widetilde{C}_1 d p )
  -\log(\epsilon/(3\widetilde{C}_1 d))}{m\log 2},
\]
which essentially means that we need to choose $p=\alpha n\log n$ with
a sufficiently large $\alpha>0$.

More precisely, considering the special case $m=1$, we have for  $n\ge 2$ and
$p=\alpha n \log_2 n$ with sufficiently large $\alpha>0$ the bound
\begin{equation}\label{alphareq}
\frac{\widetilde{C}_1 d \alpha\log_2 n }{n^{\frac{\alpha}{2} -1}} \le
\frac{\widetilde{C}_1 d \alpha}{n^{\frac{\alpha}{2}-2}} \le 
\frac{\widetilde{C}_1 d \alpha}{2^{\frac{\alpha}{2} -2}} \le 1.
\end{equation}
This means that the first term in the above estimate can, under the
assumption $p=\alpha n\log_2 n$, be further bounded by
\begin{eqnarray*}
  3(\widetilde{C}_1 d)^n p^{n-1}2^{-p} & = & 3(\widetilde{C}_1d)^n
  (\alpha n\log_2 n)^{n-1} n^{-\alpha  n} = \frac{3}{\alpha n \log_2
    n} (\widetilde{C}_1 d \alpha n^{1-\alpha} \log_2n)^n\\
  & \le& \frac{3}{\alpha n^{\frac{\alpha}{2}+1} \log_2 n},
\end{eqnarray*}
while the second term satisfies the bound
\begin{eqnarray*}
  1+3(4\widetilde{C}_1 ed^2)^np^{2n-2} 2^{p+1}  &\le& 
  7(4\widetilde{C}_1ed^2)^n (\alpha n \log_2 n)^{2n-2} n^{\alpha n}
   =  7 \frac{(4\widetilde{C}_1^2 e d^2\alpha^2 n^{2+\alpha} \log_2^2
     n)^n}{\widetilde{C}_1^n (\alpha n \log_2 n)^2}\\
   & \le & 7 \frac{(4e n^{\alpha-2} n^{2+\alpha})^n}{\widetilde{C}_1^n
     (\alpha n \log_2 n)^2} = 7 \frac{(4e
     n^{2\alpha})^n}{\widetilde{C}_1^n (\alpha n \log_2 n)^2}.
\end{eqnarray*}
Hence,  in the case of $m=1$ we have the following result. For
arbitrary $m\in\N$ we could proceed similarly but leave the details
for the reader as we are only interested in the case $m=1$ in what follows.

\begin{corollary}\label{cor:mismatch2}
  Under the assumptions of Corollary
  \ref{cor:mismatch1} with $m=1$ let $p=\alpha n \log_2 n$ with $\alpha>0$
  sufficiently large. Then, 
\[
  \|f-Q_{X,\lambda,\Lambda_n}f\|_{L_2(I^d)}  \le 
  \frac{3}{\alpha n^{\frac{\alpha}{2}+1} \log_2 n}  \|f\|_{H^1_{\mix}(I^d)} +
 \frac{7 (4 e n^{2\alpha})^n}{\widetilde{C}_1^n (\alpha n\log_2 n)^2}\|f_2\|_{L_\infty(I^d)}.
\]
\end{corollary}

Obviously, this means that the term $\|f_2\|_{L_\infty(I^d)}$ needs to
decay super exponentially to cope with the term $n^{2\alpha n}$. Also,
the assumption (\ref{alphareq}) can possibly be further optimized.

A typical application of this set-up comes from uncertainty
quantification. As this is not the primary goal of our paper, we will
not review the details but refer the reader to
\cite{Dick-etal-04-1,Kaarnioja-etal-20-1,Kuo-etal-12-1,Owen-13-1,Rieger-Wendland-26-1,Sloan-Wozniakowski-98-1,
  Sloan-etal-04-1}. For us, the only important thing is that in this
context, one tries to approximate a function from a weighted mixed
regularity Sobolev space
$
H_{{\ggamma}}^1(I^d)=\left\{ f\in L_2(I^d) : \| f\|_{H_{\ggamma}}<\infty \right\},
$
which is equipped with the norm
\[
\|f\|^2_{H_{\ggamma}^1(I^d)} = \sum_{\mfu \subseteq \cD}
\ggamma^{-1}_{\mfu}\|D^{\mfu} f((\cdot;\00)_{\mfu})\|_{L_2(I_\mfu)}^2.
\]
It is well-known that this norm is equivalent to the standard norm
$\|\cdot\|_{H^1_{\mix}(I^d)}$, see for example \cite[Theorem
  3.3]{Rieger-Wendland-26-1}. 
Moreover,  if we choose
$\Lambda=\Lambda_n$ and set $f_2=\sum_{|\mfu|\ge n+1} f_\mfu$ for an
$f\in H^1_{\mix}(I^d)$ then \cite[Corollary 3.7]{Rieger-Wendland-26-1}
yields
\begin{equation}\label{estf2}
\|f_2\|_{L_\infty(I^d)} \le C_{emb}\left( \sum_{|\mfu|\ge n+1}
2^{|\mfu|/2} \gamma_{\mfu}^{1/2}\right)
\|f\|_{H_{\ggamma}^1(I^d)},
\end{equation}
where $C_{emb}>0$ is the embedding constant from embedding
$H^{1}_{\mix}(I^d)$ into $L_\infty(I^d)$. Next, as usual, we will
assume that the weights have product form, i.e. $\ggamma_{\mfu} 
=\prod_{j\in\mfu}\gamma_j$. Then, the following prototype of error
estimate for $f_2$ can be derived, assuming super exponential decay of
the coefficients.

\begin{lemma}
 If under the above assumptions, the weights are given by $\gamma_j =
 \frac{1}{2}e^{-6j^2}$, $j\in\N_0$ then, the remaining term
 $f_2=\sum_{|\mfu|\ge n+1} f_\mfu \in H^1_{\mix,\complement \Lambda_n}(I^d)$ satisfies the bound
 \[
 \|f_2\|_{L_\infty(I^d)} \le C' e^{-n^3} \|f\|_{H^1_{\ggamma}(I^d)}
 \]
   with a constant $C'>0$ independent of $n$. 
\end{lemma}
\begin{proof}
Starting with (\ref{estf2}),  we can continue our bound by estimating
\[
\sum_{|\mfu|\ge n+1} 2^{|\mfu|/2} \gamma_{\mfu} = \sum_{k=n+1}^d
2^{k/2} \sum_{|\mfu|=k} \prod_{j\in \mfu} \gamma_j^{1/2}
= \sum_{k=n+1}^d \sum_{|\mfu|=k}\prod_{j\in \mfu}  \mu_j
\]
with $\mu_j = (2\gamma_j)^{1/2}$.

Next, if we the order the set  $\mfu=\{u_1,\ldots, u_k\}$ in the form
$u_1<u_2<\cdots <u_k$ and use $u_j\ge j$ for $1\le j\le k$, we can
write this in the form $u_j=j+\delta_j$ with $\delta_j\ge 0$. Thus, we
have
\[
\sum_{|\mfu|=k} \prod_{j\in \mfu} \mu_j  = \left(\prod_{j=1}^k
  \mu_j\right) \left( 1+ \sum_{\substack{ |\mfu|=k \\ \mfu\ne
      \{1,\ldots,k\}}}\prod_{j=1}^k \frac{\mu_{u_j}}{\mu_j}\right)
  = \left(\prod_{j=1}^k \mu_j\right) \left(1+ \sum_{0\le \delta_1\le
    \delta_2\le\cdots \le \delta_k} \prod_{j=1}^k
  \frac{\mu_{j+\delta_j}}{\mu_j}\right). 
  \]
By assumption, the weights decay super exponentially in the form
$\mu_j=\exp(-3j^2)$, which leads to 
\begin{eqnarray*}
\sum_{0\le \delta_1\le
    \delta_2\le\cdots \le \delta_k} \prod_{j=1}^k
\frac{\mu_{j+\delta_j}}{\mu_j}  &=&
\sum_{0\le \delta_1\le
    \delta_2\le\cdots \le \delta_k} \prod_{j=1}^k
e^{-3(2j\delta_j+\delta_j^2)}
\le \sum_{\ddelta\in\N_0^k} \prod_{j=1}^k
e^{-3(2j\delta_j+\delta_j^2)} \\
& = &  \prod_{j=1}^k \sum_{\delta=0}^\infty e^{-3(2j\delta+\delta^2)}
= \prod_{j=1}^k \left(1+ \sum_{\delta=1}^\infty\left(
e^{-6j}\right)^\delta e^{-3\delta^2}\right)\\
& \le& \prod_{j=1}^k \left(1+e^{-3}
\frac{e^{-6j}}{1-e^{-6j}}\right)
\le \prod_{j=1}^\infty(1+e^{-3-6j}) =:C.
\end{eqnarray*}
Plugging this all into the above estimate (\ref{estf2}), we arrive at
\begin{eqnarray*}
\|f_2\|_{L_\infty(I^d)} &\le& C_{emb}(1+C) \sum_{k=n+1}^d \prod_{j=1}^k e^{-3j^2}
\|f\|_{H_{\ggamma}^1(I^d)} \\
& = & C_{emb}(1+C) \sum_{k=n+1}^d e^{-k^3}e^{-3k^2/2}e^{-k/2}
\|f\|_{H_{\ggamma}^1(I^d)} \\
& \le & C' e^{-n^3} \|f\|_{H^1_{\ggamma}(I^d)},
\end{eqnarray*}
which is the desired estimate.
\end{proof}

As we obviously have $\|f\|_{H^1_{\mix}(I^d)} \le
\|f\|_{H^1_{\ggamma}(I^d)}$ for all $f\in H^1_{\mix}(I^d)$ we can combine
this lemma with Corollary \ref{cor:mismatch2} to derive our final
bound
\[
  \|f-Q_{X,\lambda,\Lambda_n}f\|_{L_2(I^d)}  \le \left(
  \frac{3}{\alpha n^{\frac{\alpha}{2}+1} \log_2 n}  +
 \frac{7C' (4 e n^{2\alpha})^n}{\widetilde{C}_1^n (\alpha n\log_2 n)^2}
 e^{-n^3}\right) \|f\|_{H^1_\ggamma(I^d)}.
\]

\section*{Acknowledgment}
This work was funded by the Deutsche Forschungsgemeinschaft (DFG, German
Research Foundation) –Projektnummer 452806809.

\end{document}